\documentclass[12pt]{amsart}
\usepackage[utf8]{inputenc}
\usepackage{Baskervaldx}
\usepackage[cal=cm]{mathalfa}
\usepackage{enumerate, geometry, graphicx, subcaption, amsmath, array, pdfsync, tikz, tikz-cd, color, url, float, wrapfig, comment}

\usepackage[hyperfootnotes=false, pdfencoding=auto]{hyperref}
\hypersetup{
  colorlinks= true,
  linkcolor=[RGB]{0,89,42},
  urlcolor=[RGB]{0,89,42},
  citecolor=[RGB]{0,89,42}
}

\usepackage[baskervaldx,cmintegrals,bigdelims,vvarbb]{newtxmath}
\usepackage{mathrsfs,amsthm}
\usepackage{color}
\usepackage{fullpage}

\def\Oc{{\mathcal{O}}}

\def\P{{\mathbb{P}}}
\def\cl{{\colon}}
\def\ra{{\rightarrow}}
\def\dra{{\dashrightarrow}}

\def\mfd{{\mathrm{mfd}}}

\def\PP{{\mathbb{P}}}
\def\ZZ{{\mathbb{Z}}}

\def\RR{{\mathbb{R}}}

\def\Hom{{\mathrm{Hom}}}

\def\Div{{\mathrm{Div}}}
\def\FF{{\mathbb{F}}}
\def\Supp{{\mathrm{Supp}}}
\def\CH{{\mathrm{CH}}}
\def\lw{{\mathrm{\ell w}}}

\def\Pic{{\operatorname{Pic}}}
\def\div{{\operatorname{div}}}

\title{The minimal fibering degree of a toric variety equals the lattice width of its polytope}

\theoremstyle{plain}
\newtheorem{theorem}{Theorem}
\newtheorem{proposition}[theorem]{Proposition}
\newtheorem{corollary}[theorem]{Corollary}

\newtheorem{lemma}[theorem]{Lemma}

\newtheorem{manualtheoreminner}{Theorem}
\newenvironment{manualtheorem}[1]{%
  \renewcommand\themanualtheoreminner{#1}%
  \manualtheoreminner
}{\endmanualtheoreminner}

\theoremstyle{definition}
\newtheorem{definition}[theorem]{Definition}
\newtheorem{example}[theorem]{Example}

\theoremstyle{remark}
\newtheorem{remark}[theorem]{Remark}

\numberwithin{figure}{section}
\numberwithin{theorem}{section}
\numberwithin{equation}{section}

\newcommand{\mb}{\mathbb}

\author{Audric Lebovitz}
\address{
	Mathematics Department \\
	University of Michigan \\
	Ann Arbor, MI 48109 \\
	USA
}
\email{alebovit@umich.edu}

\author{David Stapleton}
\address{
	Mathematics Department \\
	University of Michigan \\
	Ann Arbor, MI 48109 \\
	USA}
\email{dajost@umich.edu}

\thanks{During the preparation of this article, the second author was partially supported by NSF grant FRG-1952399.}

\begin{document}

\maketitle

\thispagestyle{empty}

The purpose of this paper is to compute the minimal fibering degree of a pair $(X,L)$ when $X$ is a projective toric variety with a globally generated line bundle $L$ -- we show the minimal fibering degree equals the lattice width of the polytope associated to $(X,L)$.

For any dominant, rational fibration of an $n$-dimensional projective variety $X$:
\[
\Phi\cl X\dra \PP^{n-1}
\]
let $C_\Phi\subset X$ denote the closure of a general fiber. If $L$ is any effective line bundle on $X$, the \textit{degree of $\Phi$ with respect to $L$} is
\[
\deg_L(\Phi) =  \deg([C_\Phi] \cdot L).
\]
The \textit{minimal fibering degree} of a pair $(X,L)$ is
\[
\mfd(X,L) = \min\{ \deg_L(\Phi) \mid  \Phi \cl X\dra \P^{n-1} \text{ is dominant}\}.
\]
The minimal fibering degree was introduced recently (\cite{LSU23}) in order to calculate the degree of irrationality -- a higher dimensional analogue of the gonality of a curve -- of divisors in an ambient variety. Only a few elementary examples of the minimal fibering degree have been computed. In this paper we compute the minimal fibering degree of an arbitrary projective toric variety.

\begin{manualtheorem}{A}\label{thm:A}
Let $(X,L)$ be a projective toric variety with $L$ an ample toric line bundle (or more generally, $L$ may be taken to be simply globally generated). If $P=P(X,L)$ is the lattice polytope associated to $(X,L)$ then
\[
\mfd(X,L) = \lw(P)
\]
(where $\lw(P)$ is the lattice width of the polytope $P$, described below).
\end{manualtheorem}

\noindent This gives a complete answer to \cite[Ques. 1.13(3)]{LSU23}. In the case $L$ is sufficiently ample, Theorem~\ref{thm:A} (together with \cite[Thm. A]{LSU23}) gives a new proof of the computation of the gonality of a curve in a smooth toric surface -- originally proved by Kawaguchi (\cite[Thm. 1.3]{Kaw16} see also \cite{CasCoo,CasCoo2}).

Given a lattice polytope $P\subset \RR^n$ --- i.e. the convex hull of finitely many points $x_1,\dots,x_e\in\ZZ^n$ --- the \textit{lattice width of $P$} (Definition~\ref{def:lw}) computes the minimal width of the image of $P$ under a nonzero linear projection
\[
P\ra \RR
\]
that sends lattice points to lattice points. This is an invariant of lattice polytopes of independent interest (see \cite{CharBuzFesch,CoolsLemmens,DraiMcallNill}). The lattice width has made an appearance in algebraic geometry in several places as well (\cite{CasCoo, CasCoo2, BettiTables,Kaw16,LubSch}). One feature of our proof is that we show how to compute the lattice width of a polytope $P(X,L)$ explicitly in terms of a toric resolution of singularities of $X$.

\begin{center}
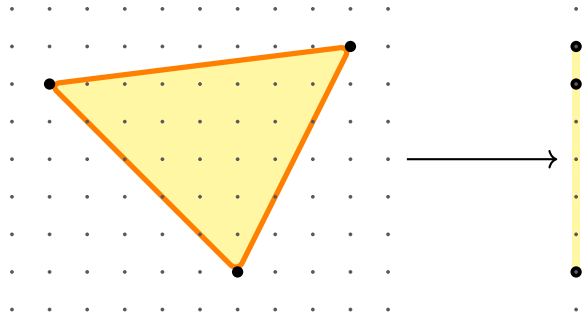
\begin{figure}[H]
\begin{tikzpicture}[scale=0.5]

  \fill[yellow!45] (6,0) -- (9,6) -- (1,5) -- cycle;
  \fill[yellow!45] (14.9,0) -- (15.1,0) -- (15.1,6) -- (14.9,6) -- cycle;

  \draw[line width=.7mm, rounded corners, orange] (6,0) -- (9,6) -- (1,5) -- cycle;
  
  \foreach \x in {0,...,10}{
    \foreach \y in {-1,...,7}{
      \fill[black!65] (\x,\y) circle (0.05);
    }
  }

  \fill[black] (6,0) circle (0.15);
  \fill[black] (9,6) circle (0.15);
  \fill[black] (1,5) circle (0.15);
  
  \fill[black] (15,0) circle (0.15);
  \fill[black] (15,6) circle (0.15);
  \fill[black] (15,5) circle (0.15);
  
  \draw[->,thick] (10.5,3) -- (14.5,3);
  \foreach \y in {-1,...,7}{
    \fill[black!65] (15,\y) circle (0.05);
  }
\end{tikzpicture}
\caption{A lattice projection that computes the lattice width of a polytope corresponding to a singular toric surface.}
\end{figure}
\end{center}

To prove the theorem we first show that any linear projection
\[
P\ra \RR
\]
sending lattice points to lattice points gives rise to an (equivariant) fibration
\[
X\dra \PP^{n-1},
\]
with fiber class a one-parameter curve, such that the degree of the fibration equals the width of the image of $P$ under the projection. This shows $\mfd(X,L) \le \lw(P)$. In the other direction (Lemma~\ref{beatthecurves}), we show that for any fibration 
\[
\Phi\cl X\dra \PP^{n-1}
\]
there is an equivariant fibration $\Phi_T:X\dra\PP^{n-1}$ with $\deg_L(\Phi_T)\le \deg_L(\Phi)$ (and the degree of any equivariant fibration is computed by the lattice width in some direction).

Throughout we work with varieties over an arbitrary algebraically closed field. All varieties are irreducible by convention.

\vspace{12pt}

\noindent\textbf{Acknowledgements.}
We are grateful to Jake Levinson for helpful conversations. This paper is the result of a REU project at the University of Michigan. We are thankful for the encouraging environment provided by the UM math department.

\section{Preliminaries on Toric Varieties} 

The purpose of this section is to recall basic facts about toric varieties including their curves, invariant divisors, and intersection numbers. We also recall the definition of a polytope of a toric pair. We include this for the convenience of the reader, but we do not claim any originality in this section. We refer the reader to Fulton's book (\cite{Fulton}) for more background. Throughout this section, let $X$ denote a projective toric variety of dimension $n$.

First we recall the standard notation. Let $N\cong \ZZ^n$ be a lattice and let $M = \Hom(N, \ZZ)$ the dual lattice (we use the notation $N_\RR=N\otimes \RR$ and $M_\RR = M\otimes \RR$). Then from a fan $\Sigma$ in $N$ one constructs a toric variety $X$ with torus $T\cong \mb{G}_m^n$ (\cite[ch. 1]{Fulton}). For each cone $\sigma\in\Sigma$ there is a associated distinguished point $x_{\sigma}\in X(k)$ (\cite[p. 28]{Fulton}). Let $O_{\sigma}$ denote the torus orbit of $x_{\sigma}$ and $V(\sigma)$ the closure of $O_{\sigma}$ in $X.$ There are bijections:
\begin{enumerate}
\item $\left\{ \begin{array}{c}\text{orbits of the action}\\T\circlearrowright X \end{array}\right\}\leftrightarrow \left\{ O_\sigma \middle| \begin{array}{c}\sigma\in \Sigma\text{ is}\\ \text{a cone}\end{array}\right\}$, and 
\item $\left\{\begin{array}{c} T-\text{invariant, closed}\\\text{subvarieties of $X$}\end{array}\right\}\leftrightarrow \left\{ V(\sigma) \middle| \begin{array}{c}\sigma\in \Sigma\text{ is}\\ \text{a cone}\end{array}\right\}.$
\end{enumerate}
The bijection (2) satisfies $\textrm{codim}(V(\sigma)) = \textrm{dim}(\sigma).$

\subsection{Cartier divisors on toric varieties} Let $D_1,\dots, D_r$ be the irreducible $T$-invariant divisors of $X$ (corresponding to the 1-dimensional cones $\tau_1,\dots,\tau_r\in \Sigma$). There is a map: 
\[
    M\ra \oplus \ZZ D_i \quad 
    \left( m \mapsto \div(\chi^m) = \sum_{i=1}^r \langle m, v_i \rangle D_i\right)
\]
(\cite[p. 61]{Fulton}), and there are standard exact sequences:
\[
0\to M\xrightarrow{\div} \oplus \ZZ D_i \to \text{Cl}(X)\to 0,
\]
and
\[
0\to M \xrightarrow{\div} \Div_T(X)\to \Pic(X)\to 0
\]
(where $\Div_T(X)$ is the subgroup of Cartier divisors in $\oplus\ZZ D_i$ \cite[p. 63]{Fulton}).

A convenient representation of the data of a torus-invariant Cartier divisor is that of its support function.

\begin{definition}[{\cite[p. 66]{Fulton}}]
Let $D$ be a torus-invariant Cartier divisor. The \textit{support function of $D$} is a piecewise linear function:
\[
\psi_D:N_{\mb{R}}\to \mb{R}
\]
defined as follows: for each cone $\sigma\in \Sigma$, $D|_{U_{\sigma}} = \div(\chi^{-m_{\sigma}})$ for some  $m_{\sigma}\in M$ and set
\[
\psi_D|_{\sigma} = \langle m_{\sigma}, \cdot \rangle.
\]
\end{definition}

\begin{remark}\label{rem:supportfunctions}
   By the definitions, $\psi_D|_N$ is integer valued. It can also be checked that $\psi_D$ is continuous. In fact, any continuous function
   \[
   \psi\cl N_{\mb{R}}\ra \RR
   \]
   which is linear on each cone and integral on $N$ arises uniquely as the support function of some torus-invariant Cartier divisor $D=a_1D_1+\cdots +a_rD_r$. Specifically, if $v_i
   \in N$ is the primitive lattice generator of the ray associated to the divisor $D_i$ then
   \[
   a_i = -\psi(v_i).
   \]
   Under this correspondence, the globally linear functions correspond to the globally principal $T$-invariant divisors.
\end{remark}

\begin{example}
    When $X = \PP^1$, the support function associated to the divisor $D = a_1 [0] + a_2 [\infty]$ is the function
    \begin{align*}
        \psi_D\cl\RR &\to \RR \\
        v &\mapsto \begin{cases}
            - a_1 v & v \ge 0 \\
            a_2 v & v \le 0
        \end{cases}
    \end{align*} and $\psi_D$ is linear if and only if $a_1 + a_2=0.$
\end{example}

\begin{definition}
If $L = \Oc(a_1D_1+\cdots a_r D_r)$ is a globally generated $T$-equivariant line bundle on $X$ then \textit{the polytope associated to $(X,L)$} is by definition
\[
P(X,L) :=\{ m\in M_\RR \mid \langle m, v_i \rangle \ge -a_i\}
\]
where $v_i\in N$ is the integral generator of the ray in $N_\RR$ associated to $D_i$. Alternatively, one can define the polytope as follows. For each torus invariant point
\[
x_1,\dots,x_m\in X^T
\]
the restriction $L|_{x_i}$ gives a 1-dimensional $T$-representation, to which we can associate a character $m_i\in M$. Then
\[
P(X,L) = \left(\begin{array}{c}
\text{the convex hull}\\
\text{of }m_i\in M
\end{array}\right)\subset M_\RR.
\]
\end{definition}

\begin{center}

\begin{figure}[H]

\begin{tabular}{ccc}

\begin{minipage}{3cm}
\begin{tikzpicture}[scale=0.75]
  \fill[yellow!30] (0,0) -- (1,0) -- (1,2) -- cycle;

  \draw[line width=.7mm, rounded corners, orange]  (0,0) -- (1,0) -- (1,2) -- cycle;
  
  \foreach \x in {-1,...,2}{
    \foreach \y in {-1,...,3}{
      \fill[black!65] (\x,\y) circle (0.05);
    }
  }

  \fill[black] (0,0) circle (0.15);
  \fill[black] (1,0) circle (0.15);
  \fill[black] (1,2) circle (0.15);
 
\end{tikzpicture}
\end{minipage}

&

\begin{minipage}{3cm}
\begin{tikzpicture}[scale=0.75]
  \fill[yellow!30] (0,0) -- (0,-1) -- (1,-1) -- (1,2) -- cycle;

  \draw[line width=.7mm, rounded corners, orange]  (0,0) -- (0,-1) -- (1,-1) -- (1,2) -- cycle;
  
  \foreach \x in {-1,...,2}{
    \foreach \y in {-2,...,3}{
      \fill[black!65] (\x,\y) circle (0.05);
    }
  }

  \fill[black] (0,0) circle (0.15);
  \fill[black] (0,-1) circle (0.15);
  \fill[black] (1,-1) circle (0.15);
  \fill[black] (1,2) circle (0.15);
 
\end{tikzpicture}
\end{minipage}

&

\begin{minipage}{4cm}
\begin{tikzpicture}[scale=0.75]
  \fill[yellow!30] (0,0) -- (3,0) -- (0,3) -- cycle;

  \draw[line width=.7mm, orange]  (0,0) -- (3,0) -- (0,3) -- cycle;
  
  \foreach \x in {-1,...,4}{
    \foreach \y in {-1,...,4}{
      \fill[black!65] (\x,\y) circle (0.05);
    }
  }

  \fill[black] (0,0) circle (0.15);
  \fill[black] (3,0) circle (0.15);
  \fill[black] (0,3) circle (0.15);
\end{tikzpicture}
\end{minipage}
\end{tabular}

\caption{Polytopes of a cone over a conic, its resolution $\FF_2$, and $(\PP^2,\Oc(3))$.}

\end{figure}

\end{center}

\begin{remark}\label{rem:globalsections}
    With $L$ as above, the lattice points of $P(X,L)$ correspond to a basis for the global sections of $L$. More precisely, each $m\in P(X,L)\cap M$ is a character which gives rise to a rational function $\chi^m$ on $X$. We have
    \[
    \div_L(\chi^m) = \div(\chi^m) + \sum_{i=1}^r a_i D_i = \sum_{i=1}^r (\langle m, v_i \rangle + a_i) D_i
    \]
    is effective since $m\in P(X,L)$ and so $\chi^m$ is a global section of $L$. Furthermore, these sections $\chi^m$ form a basis for $H^0(X,L).$ In the opposite direction, given a $T$-equivariant, globally generated line bundle $L$, the characters that appear in the $T$-representation $H^0(X,L)$ give the integral points of the polytope $P(X,L)\subset M_\RR$.
\end{remark}

\begin{remark}
There is a surjective map
\begin{equation}\label{eqn:facemap}
f\cl \left\{ \begin{array}{c}T\text{-invariant, closed}\\\text{subvarieties of }X \end{array}\right\} \ra \left\{ \begin{array}{c}\text{faces of}\\P(X,L) \end{array}\right\}
\end{equation}
that respects the partial ordering given by inclusion. It sends a divisor $D_i$ to the face:
\[
P(X,L) \cap (\langle \cdot ,v_i\rangle = 0)
\]
and this determines the map as every $T$-invariant, closed subvariety is an intersection of divisors. In the case $L$ is very ample this map is a bijection. Finallly, if $V(\sigma)$ is a closed $T$-invariant subvariety and $m\in P(X,L)$ is a lattice point then $\div(\chi^m)\in H^0(X,L)$ and
\begin{equation}\label{eqn:vanishinglocus}
V(\sigma)\subset \Supp(\div(\chi^m)) \iff m\not\in f(V(\sigma)).
\end{equation}
\end{remark}

\begin{remark}[{\cite[\S3.4]{Fulton}}]\label{rem:polytope_to_support}
If $D$ is a $T$-invariant divisor with a base point free linear system one can recover the support function $\psi_D$ from the polytope $P = P(X,\Oc(D))$ and visa versa:
\[
P = \{m\in M_{\RR} \mid \langle m, \cdot \rangle \ge \psi_D\}, \text{ and }\psi_D(v) = \min_{m\in P} \langle m, v\rangle.
\]
\end{remark}

\subsection{Toric rational maps} Let $X$ and $Y$ be toric varieties with tori $T_X$ and $T_Y$, respectively.

\begin{definition}
    A \textit{toric morphism} (resp. a \textit{toric rational map}) is a regular (resp. rational) map
    \[
    X\ra Y \quad (\text{resp. }X\dra Y)
    \]
    that extends a group homomorphism $T_X\to T_Y$.
\end{definition}

\begin{remark}
    There is a correspondence between toric rational maps $X\dra Y$ and maps of lattices $N_X\to N_Y$. Given a toric rational map, its maximum domain of definition can be determined from the map on lattices: it is the union of all $U_{\sigma}$ where $\sigma$ is a cone of $X$ mapping into a cone of $Y.$ So, under this correspondence, the regular maps (i.e. classes of toric morphisms) correspond to fan-preserving maps of lattices.
\end{remark}

\begin{definition}
    A toric rational map $\Phi:X\dra Y$ is called a \textit{toric fibration} if it is dominant.
\end{definition}

\begin{remark}
    If $\Phi:X\dra Y$ is a toric fibration then the restriction to the tori $\Phi|_{T_X}:T_X\to T_Y$ is surjective.
\end{remark}

\begin{remark}
    One could equivalently define a toric fibration to be a toric rational map $X\dra Y$ so that the the map on lattices $\overline{\Phi}:N_X\to N_Y$ satisfies $\textrm{rank}(\overline{\Phi}) = \textrm{rank}(N_Y).$
\end{remark}

\subsection{Functoriality of divisors} Let $\pi:X\to Y$ be a toric morphism. There is a pullback:
\[
\pi^*\cl \Div_T(Y) \ra \Div_T(X).
\]
First we describe how to compute this pullback in terms of support functions. Let $D\subset Y$ be a torus-invariant Cartier divisor. The morphism $\pi:X\to Y$ induces a map on lattices
\[
\overline{\pi}:(N_X)_{\mb{R}}\to (N_Y)_{\mb{R}}.
\]
The support function
\[
\psi_D:(N_Y)_{\mb{R}}\to \mb{R}.
\]
on $Y$ pulls back to a support function
\[
\psi:=\overline{\pi} \circ \psi_D:(N_X)_{\mb{R}}\to \mb{R}
\]
on $X.$ Then (by Remark~\ref{rem:supportfunctions}) there exists a unique torus-invariant Cartier divisor $D'$ on $X$ so that $\psi = \psi_{D'}$ and we have $D'= \pi^* D$.

When $L = \Oc(D)$ is globally generated, we can describe this pullback in terms of polytopes. The map $\pi$ induces a map on the dual lattices
\[
\overline{\pi}^\vee\cl M_Y\to M_X
\] 
that sends the lattice points of $P(Y,L)$ to the lattice points of $P(X, \pi^* L)$. This  corresponds to the pullback of global sections (as in Remark~\ref{rem:globalsections}). This extends to a map on the polytopes:
\[
\pi_P\cl P(Y,L) \ra P(X,\pi^*(L))
\]
In fact, as $L$ is globally generated, the map $\pi_P$ is surjective.

\begin{proposition}
If $\pi:X\to Y$ is a toric morphism and $L$ is a globally generated $T$-line bundle on $Y$ then
$\pi_P\cl P(Y,L)\ra P(X,\pi^*(L))$ is surjective.
\end{proposition}

\begin{proof}
The polytope $P(X,\pi^*(L))$ is the convex hull of its vertices, which are lattice points. As $\pi_P$ is linear and $P(Y,L)$ is convex, it suffices to show that every vertex of $P(X,\pi^*(L))$ is in the image of $\pi_P$. For every vertex $m\in P(X,\pi^*(L))$ there is a point $x\in X^T$ such that the face map (Equation~\ref{eqn:facemap}) satisfies
\[
f(x) = m.
\]
In fact, (by Equation~\ref{eqn:vanishinglocus}) $\chi^m$ is the unique character of $H^0(X,\pi^*L)$ that doesn't vanish at $x$. The map
\[
\pi\cl X\ra Y
\]
sends $x$ to a distinguished point of $Y$ corresponding to a $T$-invariant closed subvariety $V(\sigma)$ in $Y$. As $L$ is globally generated, there is necessarily a point $m'\in P(Y,L)$ such that
\[
\pi(x)\not\in\Supp(\div(\chi^{m'}))
\]
(any $m'\in f(V(\sigma))$ suffices). Therefore, $\pi^*(\chi^{m'})$ does not vanish on $x$ and thus
\[
\pi _P(m') = m.\qedhere
\]
\end{proof}

\begin{example}
    Although the map on polytopes is surjective, the map on lattice points
    \[
    \pi_P\cl P(Y,L)\cap M_Y\ra P(X,\pi^*(L))\cap M_X
    \]
    is not surjective in general. Consider the toric map
    \begin{align*}
        \pi\cl \PP^1 &\to\PP^2 \\
        [x_0,x_1]&\mapsto [x_0^3:x_0^2 x_1:x_1^3].
    \end{align*}
    If $L = \Oc_{\PP^2}(1)$ then $P(\PP^2,L)$ has 3 lattice points but $P(\PP^1,\pi^*L)$ has 4 lattice points.
    \vspace{12pt}

\begin{center}
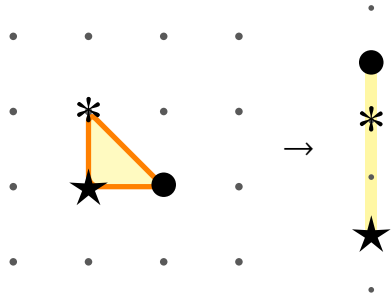
\begin{figure}[H]
\begin{tabular}{cc}

\begin{minipage}{3.5cm}
\begin{tikzpicture}[scale=1]
  \fill[yellow!30] (0,0) -- (1,0) -- (0,1) -- cycle;

  \draw[line width=.7mm, rounded corners, orange]  (0,0) -- (1,0) -- (0,1) -- cycle;
  
  \foreach \x in {-1,...,2}{
    \foreach \y in {-1,...,2}{
      \fill[black!65] (\x,\y) circle (0.05);
    }
  }

  \node[font = {\Huge\bfseries\sffamily}] at (0, 0) {$\star$};
  \node[font = {\Huge\bfseries\sffamily}] at (1, 0) {$\bullet$};
  \node[font = {\Huge\bfseries\sffamily}] at (0, 1) {$\ast$};
 
\end{tikzpicture}
\end{minipage}
$\ra$

&

\begin{minipage}{1cm}
\begin{tikzpicture}[scale=.75]
  
  \fill[yellow!45] (-.1,0) -- (.1,0) -- (.1,3) -- (-.1,3) -- cycle;

  \foreach \y in {-1,...,4}{
    \fill[black!65] (0,\y) circle (0.05);
  }

  \node[font = {\Huge\bfseries\sffamily}] at (0, 0) {$\star$};
  \node[font = {\Huge\bfseries\sffamily}] at (0, 3) {$\bullet$};
  \node[font = {\Huge\bfseries\sffamily}] at (0, 2) {$\ast$};
 
\end{tikzpicture}
\end{minipage}
\end{tabular}
\caption{A depiction of the map on polytopes from the example.}
\end{figure}
\end{center}
\end{example}

\subsection{Intersection theory of one-parameter curves and torus-invariant divisors} 

\begin{definition}
    Let $0\neq v\in N$. Then the \textit{one-parameter subgroup associated to $v$}, is the map 
    \[
    \phi_v\cl \mb{G}_m\to T
    \]
    induced by the map of lattices
    \begin{align*}
    \overline{\phi_v}\cl\ZZ &\to N \\
    1&\mapsto v.
    \end{align*}
\end{definition}

\begin{definition}
    The \textit{one-parameter curve associated to $v$}, is the induced map $\phi_v\cl \PP^1\ra X$ (which we also denote by $\phi_v$ by abuse of notation) and the cycle associated to $v$ is
    \[
    [C_v]:={(\phi_v)}_*([\PP^1]) \in Z_1(X).
    \]
\end{definition} 

For $v\in N$ nonzero and $D$ a torus-invariant Cartier divisor, we recall how to compute the intersection number $[C_v]\cdot D$. Let $\sigma_0$ be the smallest cone of $X$ containing $v$, and $\sigma_{\infty}$ be the smallest cone of $X$ containing $-v.$ Then the images of $0$ and $\infty$ are the distinguished points $x_{\sigma_0}$ and $x_{\sigma_{\infty}}$, respectively.

\begin{proposition}\label{intersectionNumbers}
    Let $v\in N$ be nonzero and $[C_v]$ the associated cycle such that the limits of $\phi_v$ at $0$ (resp. $\infty$) is $x_{\sigma_0}$ (resp. $x_{\sigma_\infty}$). If $D\subset X$ is a torus-invariant Cartier divisor with support function $\psi_D$ then the local intersection numbers at $x_{\sigma_0}$ and $x_{\sigma_\infty}$ are:
    \[
    ([C_v] \cdot D)_{x_{\sigma_0}}  = - \psi_D(v) \quad\text{ and }\quad ([C_v] \cdot D)_{x_{\sigma_\infty}} = -\psi_D(-v).
    \]
    Consequently,
    $$[C_v] \cdot D = -(\psi_D(v) + \psi_D(-v)).$$
\end{proposition}

\begin{proof} We compute $\phi_v^* \Oc_X(D)$ by pulling back the support function $\psi_D.$  The support function $\psi_D:N_{\mb{R}}\to \mb{R}$ pulls back along $\overline{\psi_v}$ to a support function
\begin{align*}
\psi_D \circ \overline{\psi_v}:\mb{R}&\to \mb{R} \\
1&\mapsto \psi_D(v).
\end{align*} 

This support function corresponds to a torus-invariant Cartier divisor $D'$ on $\PP^1$ such that $\psi_v^* \Oc_X(D) \cong \Oc_{\PP^1}(D').$ It follows that
$$([C_v] \cdot D)_{\sigma_0} = \textrm{ord}_0(D') = - \psi_{D'}(1) = -\psi_D(v)$$ and
$$([C_v] \cdot D)_{\sigma_{\infty}} = \textrm{ord}_{\infty}(D') = - \psi_{D'}(-1) = -\psi_D(-v)$$ as desired.
\end{proof}

\begin{corollary}\label{equalsone}
    If $D$ is a torus-invariant prime divisor and $v$ the primitive generator of the ray in the fan determined by $D$, then $[C_{v}] \cdot D = 1.$
\end{corollary}

\begin{proof}
In this case, $\sigma_0$ is the ray determined by $v$. The cone $\sigma_{\infty}$ contains $-v$, so it cannot contain the ray determined by $D$ as $\sigma_\infty$ is strongly convex. Then $\psi_D|_{\sigma_\infty}\equiv 0$ and thus (by Remark~\ref{rem:supportfunctions})
\[
[C_v]\cdot D = -\psi_D(v) = 1.\qedhere
\]
\end{proof}

\section{Minimal fibering degree of toric varieties}

This section contains the proof of the main theorem. To prove Theorem~\ref{thm:A} we show that the minimal fibering degree can be computed by torus equivariant fibrations. We then relate the degrees of these torus equivariant fibrations to the widths of lattice projections of the polytope of $X$.

\begin{definition}
Let $X$ be an $n$-dimensional projective variety over $k$ an algebraically closed field with line bundle $L$. Consider a dominant rational map:
\[
\Phi\cl X\dra \PP^{n-1}.
\]
Let $C_\Phi$ be the closure of a general fiber of $\Phi$. The \textit{$L$-degree of $\Phi$} is
\[
\deg_L(\Phi) = \deg([C_b]\cdot L).
\]
\end{definition}

\begin{definition}[~{\cite[Def. 1.4]{LSU23}}]
With $(X,L)$ as above, if $L$ is effective, then the \textit{minimal fibering degree} of the pair $(X,L)$ is
\[
\mfd(X,L) :=\min\{ \deg_L(\Phi) \mid \Phi \cl X\dra \PP^{n-1}\}.
\]
\end{definition}

\begin{lemma}\label{lem:vmapsexist}
Let $X$ be a smooth $n$-dimensional projective toric variety.
\begin{enumerate}
\item For every one-parameter curve $\phi_v\cl \PP^1\ra X$ there is a torus equivariant fibration
\[
\Phi_v\cl X\dra \PP^{n-1}
\]
(where $\PP^{n-1}$ has the standard torus action) such that $[C_{\Phi_v}] \equiv_{\mathrm{rat}} [C_v]\in \CH_1(X)$.
\item Likewise, for any torus equivariant fibration:
\[
\Phi\cl X\dra \PP^{n-1}
\]
there exists a one-parameter curve $\phi_v$ such that $[C_v] \equiv_{\mathrm{rat}} [C_\Phi]$.
\end{enumerate}
\end{lemma}

\begin{proof}
For (1), first consider the case when $v$ is primitive. Choose a basis $u_1,\dots,u_{n-1}$ for $v^\perp$ in $M$. These characters define a dominant map
\begin{align*}
T&\ra \mb{G}_m^{n-1}\\
t&\mapsto (\chi^{u_1}(t),\dots,\chi^{u_{n-1}}(t))
\end{align*}
that uniquely extends to a toric fibration $\Phi\cl X\dra \PP^{n-1}$. We claim that for every $T$-divisor $D\subset X$:
\[
[C_\Phi]\cdot D = [C_v]\cdot D.
\]

The closure of the fiber over $1\in \mb{G}^{n-1}$ has class $[C_v]$ (this uses that $v$ is primitive). Now, consider a general point $b\in \mb{G}^{n-1}\subset \PP^{n-1}$. Let $t\in T$ be a point in $\Phi^{-1}(b)$. So
\[
\overline{\Phi^{-1}(\Phi(t))} = C_\Phi.
\]
Then
\begin{align*}
[C_v]\cdot D & = [t^{-1}(C_\Phi)]\cdot D\\
 &= [t(t^{-1}(C_\Phi))]\cdot t (D)& \text{ (as $t$ gives an automorphism of $X$)}\\
 &= [C_\Phi] \cdot D. & \text{ (as $D$ is $T$-invariant)}
\end{align*}
As numerical equivalence and rational equivalence coincide on smooth toric varieties this implies $[C_\Phi]\equiv_{\mathrm{rat}} [C_v]$.

The case that $v=m v_0$ is not primitive (with $m>1$ and $v_0$ primitive) is resolved by postcomposing the above map
\[
T\ra \mb{G}_m^{n-1}
\]
with a degree $m$ homomorphism $\mb{G}_m^{n-1}\ra \mb{G}_m^{n-1}.$

For (2), such a torus equivariant fibration gives rise to a one-parameter subgroup (the reduced connected component of the identity in the kernel of $\Phi|_T$). Consider the induced vector $v_0\in N$. By equivariance, there is a factorization
\[
\begin{tikzcd}
T \arrow[r,"\Phi_{v_0}|_T"] \arrow[dr,swap, "\Phi|_T"]& \mb{G}_m^{n-1}\arrow[d,"\xi"]\\
&\mb{G}_m^{n-1}.
\end{tikzcd}
\]
By a computation
\[
[C_\Phi] \equiv_{\mathrm{rat}} \deg(\xi)[C_{v_0}] \equiv_{\mathrm{rat}} [C_{\deg(\xi)\cdot {v_0}}],
\]
so we may set $v=\deg(\xi) v_0$.
\end{proof}

\begin{lemma}\label{lem:nonneg}
Let $X$ be smooth, $n$-dimensional, projective variety. If 
\[
\Phi\cl X\dra \PP^{n-1}
\]
is any dominant fibration and $D\subset X$ is any effective divisor then $[C_\Phi]\cdot D\ge 0$.
\end{lemma}

\begin{proof}
Let $b\in \PP^{n-1}$ be a general point and consider the graph of $\Phi$
\[
\Gamma_\Phi\subset X\times \PP^{n-1}.
\]
Let $\pi_i$ denote the projection of $X\times \PP^{n-1}$ onto each factor. Then
\[
[C_\Phi] = (\pi_1)_* (\Gamma_\Phi \cdot X\times b).
\]
Thus
\begin{align*}
[C_\Phi]\cdot D &= (\pi_1)_* (\Gamma_\Phi \cdot X\times b) \cdot D \\
&=(\Gamma_\Phi \cdot (X\times b)) \cdot \pi_1^*D \\
&= (\Gamma_\Phi\cdot \pi_2^*(b))\cdot \pi_1^*D\\
&= (\Gamma_\Phi \cdot \pi_1^*D) \cdot \pi_2^*(b)\\
&= (\pi_2)_*(\Gamma_\Phi \cdot \pi_1^*D)\cdot b.
\end{align*}
Lastly, $\Gamma_\Phi\cdot \pi_1^*D$ is an effective divisor on $\Gamma_\Phi$ (as the map $\Gamma_\Phi\ra X$ is dominant and $\Gamma_\Phi$ is irreducible). Thus
\begin{align*}
[C_\Phi]\cdot D & = (\pi_2)_*(\Gamma_\Phi \cdot \pi_1^*D)\cdot b\\
&= \deg((\Gamma_\Phi \cap \pi_1^*D)\ra \PP^{n-1})\\
&\ge 0.\qedhere
\end{align*}
\end{proof}

\begin{lemma}\label{beatthecurves}
Let $X$ be a smooth projective toric variety and let $[C]$ be an effective curve class that meets all effective, $T$-invariant divisors non-negatively. There exists a one-parameter curve $\phi_v$ such that for all globally generated line bundles $L$:
\[
 [C]\cdot L\ge [C_v] \cdot L.
\]
\end{lemma}

\begin{example} Here we illustrate the proof of Lemma~\ref{beatthecurves} by an example. To start we are given a curve $C$ with prescribed intersection numbers with the torus-invariant divisors, as in the left figure. Consider any irreducible $T$-divisor $D$ that has positive intersection number with $C$ (in this case the top right divisor). From this divisor we construct a one-parameter curve $C_v$. This has intersection multiplicity 1 with $D$ and meets exactly one other stratum of the toric variety (in our example, the intersection of the bottom and left curves, as illustrated in the middle figure). Given any globally generated line bundle there is a section that does not vanish at the distinguished point of the stratum (in this example we take the section $m$ in the right figure), the corresponding divisor is particularly convenient to intersect with $[C_v]$ as most of the terms vanish.
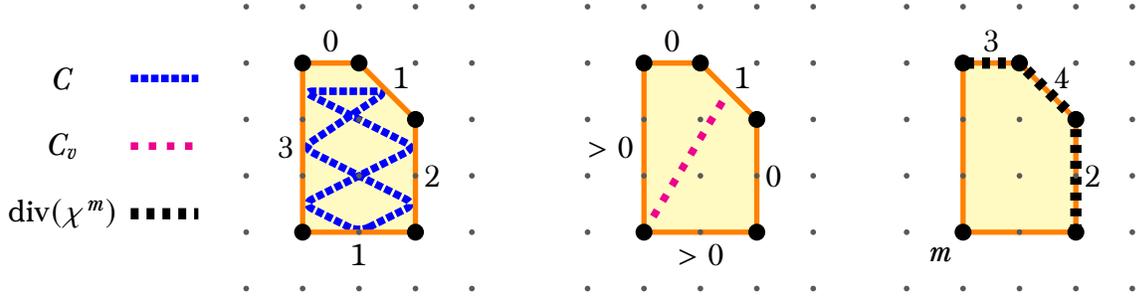
\begin{center}
\begin{figure}[H]
\begin{tabular}{ p{2.5cm} p{4.5cm} p{4cm} p{4cm}}

\begin{minipage}{2.5cm}
\begin{tikzpicture}[scale=0.6]
  \draw[line width=1.5mm, dashed, black] (0,0)--(1.5,0);
  \node at (-1.5,0) {$\div(\chi^m)$};
  \draw[line width=1mm, loosely dotted, magenta] (0,1.5)--(1.5,1.5);
  \node at (-1.5,1.5) {$C_v$};
  \draw[
  rounded corners, densely dotted, line width=1mm, blue] (0,3)--(1.5,3);
  \node at (-1.5,3) {$C$};
\end{tikzpicture}
\end{minipage}
&

\quad\begin{minipage}{4cm}
\begin{tikzpicture}[scale=0.75]
  \fill[yellow!30] (0,0) -- (2,0) -- (2,2) -- (1,3) -- (0,3) -- cycle;

  \draw[rounded corners, densely dotted, line width=1mm, blue] (1,0) -- (2,.5) -- (0,1.5) -- (1.5,2.5) -- (0,2.5) -- (2,1.5) -- (0,.5) -- cycle;

  \draw[line width=.7mm, orange] (0,0) -- (2,0) -- (2,2) -- (1,3) -- (0,3) -- cycle;
  
  \foreach \x in {-1,...,3}{
    \foreach \y in {-1,...,4}{
      \fill[black!65] (\x,\y) circle (0.05);
    }
  }

  \fill[black] (0,0) circle (0.15);
  \fill[black] (2,0) circle (0.15);
  \fill[black] (2,2) circle (0.15);
  \fill[black] (1,3) circle (0.15);
  \fill[black] (0,3) circle (0.15);

  \node at (2.3,1) {$2$};
  \node at (1.75,2.75) {$1$};
  \node at (1,-.4) {$1$};
  \node at (-.3,1.5) {$3$};
  \node at (.5,3.4) {$0$};
\end{tikzpicture}
\end{minipage}

&

\begin{minipage}{4.5cm}
\begin{tikzpicture}[scale=0.75]
  \fill[yellow!30] (0,0) -- (2,0) -- (2,2) -- (1,3) -- (0,3) -- cycle;

  \draw[line width=1mm, loosely dotted, magenta] (0,0) -- (1.5,2.5);

  \draw[line width=.7mm, orange] (0,0) -- (2,0) -- (2,2) -- (1,3) -- (0,3) -- cycle;
  
  \foreach \x in {-1,...,3}{
    \foreach \y in {-1,...,4}{
      \fill[black!65] (\x,\y) circle (0.05);
    }
  }

  \fill[black] (0,0) circle (0.15);
  \fill[black] (2,0) circle (0.15);
  \fill[black] (2,2) circle (0.15);
  \fill[black] (1,3) circle (0.15);
  \fill[black] (0,3) circle (0.15);

  \node at (2.3,1) {$0$};
  \node at (1.75,2.75) {$1$};
  \node at (1,-.4) {$>0$};
  \node at (-.6,1.5) {$>0$};
  \node at (.5,3.4) {$0$};
 
\end{tikzpicture}

\end{minipage}

&

\begin{minipage}{4cm}
\begin{tikzpicture}[scale=0.75]
  \fill[yellow!30] (0,0) -- (2,0) -- (2,2) -- (1,3) -- (0,3) -- cycle;

  \draw[line width=.7mm, orange] (0,0) -- (2,0) -- (2,2) -- (1,3) -- (0,3) -- cycle;

  \draw[line width=1.5mm, dashed, black] (2,0) -- (2,2) -- (1,3) -- (0,3);
  
  \foreach \x in {-1,...,3}{
    \foreach \y in {-1,...,4}{
      \fill[black!65] (\x,\y) circle (0.05);
    }
  }

  \fill[black] (0,0) circle (0.15);
  \fill[black] (2,0) circle (0.15);
  \fill[black] (2,2) circle (0.15);
  \fill[black] (1,3) circle (0.15);
  \fill[black] (0,3) circle (0.15);

  \node at (2.3,1) {$2$};
  \node at (1.75,2.75) {$4$};
  \node at (.5,3.4) {$3$};
  \node at (-.4,-.4) {$m$};
\end{tikzpicture}
\end{minipage}
\end{tabular}
\end{figure}
\end{center}

\vspace{-36pt}

\noindent By direct comparison we complete the proof, in this example we have:
\begin{align*}
[C]\cdot L =& [C]\cdot \div(\chi^m)\\
=& 1\cdot 0 + 2\cdot 2+1\cdot 4+0\cdot 3+3\cdot 0\\
\ge & (>0)\cdot 0 + 0 \cdot 2+1\cdot 4+0\cdot 3+(>0)\cdot 0\\
= &[C_v] \cdot L.
\end{align*}

\end{example}

\begin{proof}[Proof of Lemma~\ref{beatthecurves}]
Let $D_1\dots D_r$ be the $T$-invariant divisors. As the curve class $[C]$ corresponds to an effective curve and $X$ is projective, it cannot intersect all the effective $T$-divisors to multiplicity $0$. Assume that
\[
[C]\cdot D = a>0
\]
for one such $T$-invariant prime divisor $D$.

Let $v\in N$ be the integral generator of the ray corresponding to $D$. This gives rise to a one-parameter curve
\[
\phi_v:\PP^1\ra X.
\]
Let $\sigma$ be the smallest cone containing $-v\in N$. Then the image of $0\in \PP^1$ is the distinguished point of $D$ and the image of $\infty \in \PP^1$ is the distinguished point of $\sigma$. These are the only intersection points of $C_v$ with the $T$-invariant divisors of $X$.

As $L$ is globally generated there is a character $m\in M$ such that $\chi^m\in \Gamma(X,L)$ and $\chi^m$ does not vanish at the distinguished point of $\sigma$. Then
\[
\div(\chi^m) =b_1 D_1+\cdots bD+\cdots+b_r D_r \quad (b,b_i \ge 0).
\]
As $\chi^m$ does not vanish at the distinguished point of $\sigma$, it follows that for every $D_i$ such that $x_\sigma\in D_i$ the coefficient $b_i=0$. So, as $D$ is the unique prime divisor containing its distinguished point, by Corollary~\ref{equalsone}:
\[
[C_v] \cdot L = [C_v]\cdot (b D) = b.
\]
Therefore:
\begin{align*}
[C]\cdot L =& [C]\cdot \div(\chi^m)&\\
=& [C] \cdot (b_1 D_1+\cdots +bD+\cdots +b_r D_r)&\\
= & [C] \cdot (b_1 D_1)+\cdots + [C]\cdot (bD) +\cdots +[C]\cdot (b_r D_r)&\\
\ge & [C] \cdot (bD) &\text{(by assumption on $[C]$)}\\
=& b\cdot a  \ge b\cdot 1 =  [C_v]\cdot L.&\qedhere
\end{align*}

\end{proof}

\begin{remark}
Interestingly, for any curve class $0\ne [C]\in \CH_1(X)$ such that $[C]\cdot D\ge 0$ for all effective $T$-invariant divisors $D$, Payne constructs (\cite[Prop. 2]{Payne}) a family of curves with class $[C]$ that sweeps out $X$.
\end{remark}

\begin{definition}\label{def:lw}
Let $P\subset M_\RR$ be a lattice polytope (that is, the convex hull of finitely many point in $M$). For any nonzero $v\in N$, the \textit{lattice width of $P$ in the direction of $v$} is 
\[
\lw_v(P) := \text{width}\left(\{ \langle x, v \rangle\in \RR | x \in P \}\right).
\]
In other words, $v$ induces a linear projection (mapping $M$ to $\ZZ$) from $P$ to $\RR$, and $\lw_v(P)$ is the width of the image. The \textit{lattice width of $P$} is simply:
\[
\lw(P) := \min \{ \lw_v(P) \mid 0\ne v \in N\}.
\]
\end{definition}

\begin{example}
For the figure in the introduction, the lattice width is 6 (which the map in the introduction realizes). To argue this, note that there are 8 lattice points on a line in $P$ that are all identified under this horizontal projection. Any projection that does not identify these points, must have width at least 7.
\end{example}

\begin{proposition}\label{prop:lwcomp}
    Let $D$ be a globally generated torus-invariant divisor and $v\in N.$ Then
\[
 [C_v] \cdot D = \lw_v(P(X,\Oc_X(D)).
\]
\end{proposition}

\begin{proof}
    By Remark~\ref{rem:polytope_to_support}, for each $u\in N$ we have $$\psi_D(u) = \min_{x\in P(X,\Oc_X(D))} \langle x, u\rangle.$$ It follows that
    $$\lw_v(P(X,\Oc_X(D)) = \max_{x\in P(X,\Oc_X(D))} \langle x, v\rangle - \min_{x\in P(X,\Oc_X(D))} \langle x, v\rangle = - (\psi_D(v) + \psi_D(-v)).$$ So, we are done by Proposition~\ref{intersectionNumbers}.
\end{proof}

\begin{proof}[Proof of Theorem~\ref{thm:A}]
Assume $L=\Oc_X(D)$ for some $T$-invariant divisor $D$. If $X$ is not smooth, take a projective, toric resolution of singularities of $X$ to obtain
\[
\mu\cl X'\ra X.
\]
Now the pair $(X',\mu^*L)$ satisfies
\[
P(X',\mu^*L) = P(X,L),
\]
and $\mu^*L$ is globally generated. Finally for any rational fibration:
\[
\Phi\cl X\dra \PP^{n-1}
\]
there is an induced rational fibration
\[
\Phi'\cl X'\dra \PP^{n-1} 
\]
and $\Phi'$ satisfies
\[
\mu_*[C_{\Phi'}] = [C_\Phi].
\]
So by the projection formula:
\[
\mfd(X,L) = \mfd(X',\mu^*L).
\]
Therefore, it suffices to prove Theorem~\ref{thm:A} in the case $X$ is smooth.

By Lemma~\ref{lem:vmapsexist} for any nonzero $v\in N$, there is a torus equivariant fibration:
\[
\Phi_v \cl X\dra \PP^{n-1}
\]
such that $[C_{\Phi_v}] \equiv_{\mathrm{rat}} [C_v]$. By Proposition~\ref{prop:lwcomp},
\[
\deg_L(\Phi_v) = [C_v]\cdot D = \lw_v(P(X,L)).
\]
Thus, by the definitions:
\[
\mfd(X,L) \le \lw(P(X,L)).
\]

On the other hand, by Lemma~\ref{lem:nonneg}, if
\[
\Phi\cl X\dra \PP^{n-1}
\]
is any fibration, then $[C_\Phi]\cdot E\ge 0$ for any effective divisor $E$ on $X$. Then by Lemma~\ref{beatthecurves}, there exists a $v\in N$ such that
\[
[C_v]\cdot D \le [C_\Phi]\cdot D = \deg_L(\Phi).
\]
Therefore (by Proposition~\ref{prop:lwcomp}):
\[
\lw(P(X,L)) \le \lw_v(P(X,L)) = [C_v]\cdot D\le \deg_L(\Phi).
\]
As $\Phi$ is arbitrary this gives $\lw(P(X,L))\le \mfd(X,L)$ which completes the proof.
\end{proof}

\begin{remark}
The proof of Theorem~\ref{thm:A} shows that if $X$ is smooth there is a finite list of rational fibrations:
\[
\mathcal{S} = \{ \Phi_v\cl X\dra \PP^{n-1} \mid v\in \Sigma\text{ generates a one-dimensional cone}\}
\]
such that for any globally generated line bundle $L$:
\[
\mfd(X,L) = \min\{ \deg_L(\Phi_v) \mid \Phi_v\in \mathcal{S} \}.
\]
Similarly one can compute the lattice width of $P(X,L)$ by minimizing over generators of the one-dimensional cones of $\Sigma.$
\end{remark}

\bibliographystyle{siam} 
\bibliography{Biblio}

\end{document}